\documentclass[11pt]{amsart}
\usepackage{amsmath,amsthm,amssymb,amsfonts,amscd,palatino}
\usepackage{amsxtra}
\usepackage{eucal}
\usepackage{mathrsfs}
\usepackage{color, graphics}
\usepackage[all]{xy}
\usepackage{hyperref}
\usepackage{pstricks,pst-plot}
\usepackage{cases}
\usepackage{arydshln}

\hypersetup{
  colorlinks,
  citecolor=blue,
  linkcolor=magenta,
  urlcolor=olive}

\headsep=1cm \numberwithin{equation}{section}
\SelectTips{cm}{11}

\newcommand \ini {\ensuremath{\mathrm{in}}}
\newcommand \ND {\ensuremath{\mathrm{ND}(1)}}
\newcommand \NDbf {\ensuremath{\mathbf{ND(1)}}}
\newcommand \Tor {\ensuremath{\mathrm{Tor}}}

\newcommand \reg {\mathrm {reg}}
\newcommand \sat {\mathrm {sat}}

\newcommand \codim {\ensuremath{\mathrm{codim}}}

\newcommand \depth {\ensuremath{\mathrm{depth}}}

\newcommand \cha {\ensuremath{\mathrm{char}}}
\def\P{{\mathbb P}}

\newcommand \st[1] {\stackrel{#1}{\longrightarrow}}

\newcommand \Gin {\ensuremath{\mathrm{Gin}}}

\def\ds{\displaystyle}
\newcommand \rrlap[1]{\hbox to 0pt{#1}}

\newcommand \ti[1]{\textit{#1}}
\newcommand \trm[1]{\textrm{#1}}

\newcommand \tbf[1]{\textbf{#1}}

\newcommand{\s}{\mathcal}

\theoremstyle{theorem} 
\newtheorem{Thm}{Theorem}[section]

\newtheorem{Prop}[Thm]{Proposition}
\newtheorem{Coro}[Thm]{Corollary}
\newtheorem{Lem}[Thm]{Lemma}

\theoremstyle{definition}
\newtheorem{Def}[Thm]{Definition}
\newtheorem{Ex}[Thm]{Example}

\newtheorem{Remk}[Thm]{Remark}

\numberwithin{equation}{section}

\begin{document}

\title[1-normality of general linear sections and some graded Betti numbers of 3-regular schemes]{Linear normality of general linear sections and some graded Betti numbers of 3-regular projective schemes}
\author[J.\ Ahn and K.\ Han] {Jeaman Ahn and Kangjin Han${}^{*}$}
\address{Department of Mathematics Education, Kongju National University, 182, Shinkwan-dong, Kongju, Chungnam 314-701, Republic of Korea}
\email{jeamanahn@kongju.ac.kr}
\address{School of Mathematics, Korea Institute for Advanced Study (KIAS),
85 Hoegiro, Dongdaemun-gu, Seoul 130--722, Korea}
\email{kangjin.han@kias.re.kr}
\thanks{${}^{*}$ Corresponding author.}
\thanks{The first author was supported by the research grant of the Kongju National University in 2013 (No. 2013-0535). The second author was supported by Basic Science Research Program through the National Research Foundation
of Korea (NRF) funded by the Ministry of Education (grant No. 2012R1A1A2038506). }



\begin{abstract}
In this paper we study graded Betti numbers of any nondegenerate 3-regular algebraic set $X$  in a projective space $\mathbb P^{n}$. More concretely, via Generic initial ideals (Gins) method we mainly consider `tailing' Betti numbers, whose homological index is not less than $\codim(X,\mathbb P^{n})$. For this purpose, we first introduce a key definition `$\mathrm{ND(1)}$ property', which provides a suitable ground where one can generalize the concepts such as `being nondegenerate' or `of minimal degree' from the case of varieties to the case of more general closed subschemes and give a clear interpretation on the tailing Betti numbers. Next, we recall basic notions and facts on Gins theory and we analyze the generation structure of the reverse lexicographic (rlex) Gins of 3-regular $\mathrm{ND(1)}$ subschemes. As a result, we present exact formulae for these tailing Betti numbers, which connect them with linear normality of general linear sections of $X\cap \Lambda$ with a linear subspace $\Lambda$ of dimension at least $\codim(X,\mathbb P^{n})$. Finally, we consider some applications and related examples. 
\end{abstract}

\keywords{graded Betti numbers, generic initial ideals, linear normality, 3-regular scheme}
\subjclass[2010]{Primary:14N05; Secondary:13D02}
\maketitle
\tableofcontents \setcounter{page}{1}

\section{Introduction}\label{sec_into}

Throughout this paper, we will work with a closed subscheme $X$ of codimension $e$ in $\mathbb P^{n}$ over an algebraically
closed field $k$ of $\cha(k)=0$. Let $I_X$ be $\bigoplus_{m=0}^{\infty} H^0(\mathcal I_{X/\P^n}(m))$, the defining ideal of $X$ in the polynomial ring $R=k[x_0, x_1 ,\ldots, x_{n}]$. We mean (co)dimension and degree of $X\subset\P^n$ by the definitions deducing from Hilbert polynomial of $R/I_X$ (see Convention in section \ref{sec_gins}).

Suppose that $X$ is $3$-regular (in the sense of Castelnuovo-Mumford), that is
\[H^i(\mathcal{I}_{X/\P^n}(3-i))=0~\trm{for all $i\ge1$}~.\]
The category of $3$-regular subschemes is quite large. It contains many known examples such as (embedded) curves of high degree, secants of rational normal scrolls, and del Pezzo varieties. Until now, $3$-regularity seems not to be well-understood. For instance, in contrast with $2$-regularity, there is no classification for $3$-regularity even in the category of varieties. So, it is worthwhile to investigate graded Betti numbers of $3$-regular subschemes.

When $X$ is $3$-regular, we could present the Betti table of $X$ with 3 rows (e.g. see \cite[chap. 4]{E2}) and divide the whole table into four quadrants (by \ti{dashed} lines) as in Figure \ref{fig_betti}.

\begin{figure}[!htb]
\texttt{\begin{tabular}{c|cccccccc:ccccc|}
   &0 &1 &2 & $\cdots$ &a & a+1 & $\cdots$ & e & e+1 &$\cdots$& n-1 & n & n+1\\ \hline 
0 &1 &0 & 0 & $\cdots$ & 0  &    0     & $\cdots$& 0  & 0 & $\cdots$ & 0 & 0  & 0 \\ 
1 &0 &$\beta_{1,1}$ & $\beta_{2,1}$ &$\cdots$ &$\beta_{\mathtt{a},1}$  & $\beta_{\mathtt{a}+1,1}$ &$\cdots$ & $\beta_{\mathtt{e},1}$ &0 &$\cdots$ &0&0&0\\  \hdashline
2 &0 &0 &0 & $\cdots$ & 0 &$\beta_{\mathtt{a}+1,2}$& $\cdots$ &$\beta_{\mathtt{e},2}$ &$\beta_{\mathtt{e}+1,2}$ & $\cdots$&$\beta_{\mathtt{n}-1,2}$& $\beta_{\mathtt{n},2}$& 0\\
\end{tabular}}
\caption{A typical \textsf{Betti table} of any nondegenerate 3-regular projective scheme $X$ in $\P^n$. Here $\beta_{i,j}$ denotes the \ti{graded Betti number} of $X$ $\dim_k\Tor^R_i(R/I_X,k)_{i+j}$, $\mathtt{e}=\codim(X,\P^n)$. $\mathtt{a}=\mathtt{a}(X)$ is often called \ti{Green-Lazarsfeld index} of $X\subset\P^n$.}
\label{fig_betti}
\end{figure}

Once $X$ is nondegenerate (i.e. $I_X$ has no linear forms), in the first row all the Betti numbers except $\beta_{0,0}$ are zeros. On the other entries of $1^{\mathrm{st}}$-quadrant, it was first known by Green's $K_{p,1}$-theorem for compact complex manifolds in \cite{G2} that they are also all \ti{zeros}; we prove this in a more general category using our method (see Theorem \ref{thm:20140212-1} and Remark \ref{remk_NP}). In $2^{\mathrm{nd}}$-quadrant the part corresponding $\beta_{1,1},\beta_{2,1},\cdots$ is called \ti{linear strand} of the table and many authors have studied on this subject of linear syzygies (e.g. see \cite[chap. 8,9]{E2} for an overview of its connection toward Green and Green-Lazarsfeld conjectures of curves or see \cite{HK2} for its relevance to classification of varieties of small degree). For $3^{\mathrm{rd}}$-quadrant, there is a nice geometric interpretation related to the existence of a \ti{degenerate secant plane} $X\cap \Lambda$, which is a finite scheme with $length(X\cap \Lambda)>\dim\Lambda+1$ for a linear subspace $\Lambda$ of dimension $\le e$ (see e.g.\cite[theorem 2]{GL}, \cite[theorem 1.1]{EGHP1}). Recently, the first author and S. Kwak obtained a more refined result on Betti numbers in the $3^{\mathrm{rd}}$-quadrant in this viewpoint (see \cite[theorem 1.2 and example 3.4]{AK}). 

In this paper, to complete the picture of Betti tables of $3$-regular schemes, we focus on the interpretation of the $\beta_{e,2},\beta_{e+1,2},\cdots,\beta_{n,2}$ (we call these \ti{tailing Betti numbers}) in the remaining part, i.e. \ti{fourth} quadrant.

{\bf $\NDbf$ property\quad} As usual, we would like to assume $X$ to be nondegenerate. But, to generalize what we have and expect in varieties to more general schemes, in many cases it is not enough. Therefore, let us begin this study by introducing a definition, which makes things clear, as follows:
\begin{Def}[$\ND$ property]\label{def_ND1}
Let $k$ be any field as above. We say that a closed subscheme $X\subset\P_k^n$ satisfies $\ND$ \ti{property} if for a \ti{general} linear section $\Lambda$ of each dimension $\geq e$
\[\trm{$X\cap \Lambda$ is nondegenerate (i.e. $H^0(\mathcal{I}_{X\cap \Lambda/\Lambda}(1))=0$).}\] 
\end{Def}

Note that by definition $X\subset\P^n$ itself is nondegenerate if $X$ satisfies $\ND$ property. We also remark that every general linear section of $X\cap \Lambda$ also has the same property in case that $X$ has $\ND$ property.

This definition of $\ND$ property is very natural. For example, every nondegenerate integral (i.e. irreducible and reduced) variety has $\ND$ property. In fact, the category of closed subschemes having $\ND$ property is quite large; we do not need to assume our subscheme $X$ to be \ti{irreducible, equi-dimensional,} or even \ti{reduced} necessarily (see Example \ref{ND_ex}). Further, this is also proper in the sense that via this property one could generalize the concepts such as `being nondegenerate' and `of minimal degree' (i.e. degree is no less than codimension+1) to more general reducible subschemes and give a very clear meaning of the tailing Betti numbers (see Proposition \ref{basic_ineq}, Theorem \ref{thm:20140212-1} and Theorem \ref{main}).

\begin{Ex}[Subschemes with $\ND$ property]\label{ND_ex}
We list the following examples of subschemes with $\ND$ property:
\begin{itemize}
\item[(a)] In case of $X$ being a nondegenerate integral variety, $\ND$ property of $X$ can be deduced from Bertini type theorem. More generally, if any algebraic set $X$ has a nondegenerate component, then $X$ also satisfies $\ND$ property.
\item[(b)] (Case of all the components degenerate) There are algebraic sets satisfying $\ND$ whose all components are degenerate. For examples, consider the following saturated ideal 
\[I=(L_1L_2L_3, L_1L_2L_4, L_1L_4L_5, L_4L_5L_6, L_1L_2L_7),\]
where $L_i$ is a generic linear form for each $i=1,\ldots, 7$ in $R=k[x_0,x_1,x_2,x_3]$. Then the algebraic set $X\subset \P^3$ defined by the ideal $I$ is the union of $5$ lines and one point such that its minimal free resolution is given by 
\[{\tiny\texttt{
 \begin{tabular}{c|ccccccccccccccccccccccccccccc}
           &0 & 1 & 2 & 3 \\[1ex]
        \hline\\
        0 & 1 & 0 &  0 & 0  \\[1ex]
        1 & 0 & 0 &  0 & 0  \\[1ex]
        2 & 0 & 5 &  5 & 1  \\[1ex]
  \end{tabular}
}}~. \]
Note that all the components of $X$ are degenerate in $\P^3$. But, a general hyperplane section $X\cap H$ is a set of $5$-points in $\P^2$ and it is nondegenerate. Hence, $X$ satisfies $\ND$ property.
\item[(c)] (Non-reduced case) A closed subscheme satisfying $\ND$ is not necessarily reduced. A simple example is a non-reduced scheme $X$ in $\P^3$ defined by the following saturated monomial ideal
$$I_X=(x_0^3, x_0^2x_1, x_0x_1^2, x_1^3, x_0^2x_2)\subset R=k[x_0,x_1,x_2,x_3],$$
then $X$ is an one-dimensional nondegenerate closed subscheme in $\P^3$ (i.e. $\codim(X,\P^3)=2)$. Since $I_X$ is a Borel fixed monomial ideal, we can verify $\Gin(I_{X\cap H/H})=(x_0^2, x_0x_1^2, x_1^3)$, which has no linear form. So does $I_{X\cap H/H}$ (i.e. $X\cap H$ is nondegenerate in $H$). Hence $X$ satisfies $\ND$.
\end{itemize}
\end{Ex}

We also give some examples of subschemes without $\ND$ property (see Example \ref{non-ND_ex}).

{\bf Main results\quad} Now, we present our main results on Betti numbers of $\ND$ subschemes:
\begin{Thm}[Theorem \ref{thm:20140212-1} and \ref{main}]\label{main_thm}
Let $X$ be any closed subscheme of codimension $e$ in $\P^n$ satisfying $\ND$ property and let $I_{X}$ be the (saturated) defining ideal of $X$. Then, we have
\begin{itemize}
\item[(a)] ($K_{p,1}$-theorem for subschemes with $\ND$ property) 
$$\beta_{i,1}(R/I_X)=0~\trm{for any $i>e$.}$$
\item[(b)] (Formulae for tailing Betti numbers) Suppose that $X$ is $3$-regular. For each $i\geq e$,
\[\beta_{i,2}(R/I_X)=\sum_{\alpha=i}^{n} \binom{\alpha}{i}\,\dim H^1(\mathcal I_{X\cap \Lambda^{\alpha}/ \Lambda^{\alpha}}(1))~,\]
where $\Lambda^{\alpha}$ is a general linear space of dimension $\alpha$.
\end{itemize}
\end{Thm}

In particular, the tailing Betti numbers of 3-regular scheme $X$ with $\ND$ property depend only on the \ti{linear normality} (or \ti{1-normality}) of general linear sections $X\cap \Lambda$ with $\dim\Lambda\ge e$. 

Here is an immediate corollary of Theorem \ref{main_thm} as follows:

\begin{Coro}\label{main_cor}
Let $X$ be any closed subscheme of codimension $e$ in $\P^n$ satisfying $\ND$ property and let $I_{X}$ be the (saturated) defining ideal of $X$. Suppose that $X$ is $3$-regular. Let $\mathbf{b}$ and $\mathbf{h}$ be two row vectors consisting of consecutive tailing Betti numbers and 1-normalities of general linear sections of $X$ as below:
\begin{align*}
&\mathbf{b}:=[\beta_{e,2}(R/I_X),\beta_{e+1,2}(R/I_X),\cdots,\beta_{n-1,2}(R/I_X),\beta_{n,2}(R/I_X)]\\
&\mathbf{h}:=[h^1(\mathcal{I}_{X\cap \Lambda^{e}/\Lambda^{e}}(1)),h^1(\mathcal{I}_{X\cap \Lambda^{e+1}/\Lambda^{e+1}}(1)),\cdots,h^1(\mathcal{I}_{X\cap \Lambda^{n-1}/\Lambda^{n-1}}(1)),h^1(\mathcal{I}_{X/\P^n}(1))]~.
\end{align*}
Then, we have identities
\begin{equation}\label{formula_h}
\trm{$\mathbf{b}^{T}=\Xi(n,e)\cdot\mathbf{h}^{T}$ and $\mathbf{h}^{T}=\Xi(n,e)^{-1}\cdot\mathbf{b}^{T}$~,}
\end{equation}
where $\mathbf{b}^{T}$ (resp. $\mathbf{h}^{T}$) is the transpose of $\mathbf{b}^{T}$ (resp. of $\mathbf{h}^{T}$) and $\Xi(n,e)$ is the invertible $(n-e+1)\times(n-e+1)$-matrix such as
{\small\begin{align}\label{matXi_0}
\Xi(n,e)=&\left[\begin{array}{*6{c}}
\displaystyle{e\choose e}&\displaystyle{e+1\choose e}&\displaystyle{e+2\choose e}&\cdots&\displaystyle{n-1\choose e}&\displaystyle{n\choose e}\\
0&\displaystyle{e+1\choose e+1}&\displaystyle{e+2\choose e+1}&\cdots&\displaystyle{n-1\choose e+1}&\displaystyle{n\choose e+1}\\
0&0&\displaystyle{e+2\choose e+2}&\cdots&\displaystyle{n-1\choose e+2}&\displaystyle{n\choose e+2}\\
\vdots&\vdots&\vdots&\ddots&\vdots&\vdots\\
0&0&0&\cdots&\displaystyle{n-1\choose n-1}&\displaystyle{n\choose n-1}\\
0&0&0&\cdots&0&\displaystyle{n\choose n}
\end{array}\right]~.
\end{align}}
In other words, 1-normalities of general linear sections $X\cap \Lambda$ with $\dim\Lambda\ge e$ uniquely determine the tailing Betti numbers and vice versa. 
\end{Coro}

Using these tailing Betti numbers $\beta_{e,2}(R/I_X),\cdots,\beta_{n,2}(R/I_X)$, we can determine $P_X(t)$, Hilbert polynomial of $X\subset\P^n$ as follows:

\begin{Coro}[Corollary \ref{app_hilb}]
Let $X$ be any $3$-regular $r$-dimensional connected algebraic set in $\P^n$ satisfying $\ND$ property. Set $e=n-r$, the codimension. Then, $P_X(t)$, Hilbert polynomial of $X\subset\P^n$ can be computed as
\begin{align*}
P_X(t)&=\ds\left\{e+1+\sum^n_{i=e}(-1)^{i-e}{i\choose e}\beta_{i,2}(R/I_X)\right\}{t+r-1\choose r}\\
&+\sum^{r-1}_{i=0}\left[\left\{1-\beta_{n-i-1,2}(R/I_X)+\sum^n_{j=n-i}(-1)^{j-n+i}{j+1\choose n-i}\beta_{j,2}(R/I_X)\right\}{t+i-1\choose i}\right]~.
\end{align*}
In particular, the degree of $X\subset\P^n$ can be given by
\begin{align*}
e+1+\sum^n_{i=e}(-1)^{i-e}{i\choose e}\beta_{i,2}(R/I_X)
\end{align*}
and the arithmetic genus $p_a(X)$ in case of $X$ being a curve and the irregularity $q(X)$ in case of $X$ being a surface can be given by $\beta_{n-1,2}(R/I_X)-{n+1\choose n}\beta_{n,2}(R/I_X)$.\end{Coro}

For this purpose, in section \ref{sec_gins} we briefly review Generic initial ideals (Gins) theory and develop some combinatorial methods to study tailing Betti numbers using Gins with \ti{reverse lexicographic order (rlex)}. In section \ref{sect_str}, we first regard some properties of $\ND$ subschemes, next consider results which connect these Betti numbers with those of rlex Gins, and in the end we prove an important proposition which describes the generation structure of the rlex Gins of 3-regular subschemes virtually. Finally, in section \ref{sect_main}, we give a proof on our main result and consider some applications and examples (see Example \ref{final_ex}). Note that, in case of arithmectically Cohen-Macaulay (ACM) or arithmectically Gorenstein (AG), there exists only one tailing Betti number $\beta_{e,2}(R/I_X)$ and the meaning is well-known. Since we consider a general situation which is not necessarily ACM nor AG, our approach uses combinatorial rlex Gins techniques instead of using duality theorems and our results can cover more general cases.

{\bf Acknowledgement} Both authors are grateful to Professor Sijong Kwak for his supports. He invited both of us to Algebra Structures and its Applications Research Center (ASARC) in Korea Advanced Institute of Science and Technology (KAIST) and gave a nice chance to start our collaborations.

\section{Generic initial ideals : A brief review}\label{sec_gins}

{\bf Convention} We are working on the following convention:
\begin{itemize}
\item ($d$-normality) \ti{Linear normality} of $X\subset\P^n$ means $\dim_k H^1(\mathcal{I}_{X/\P^n}(1))$ which is cokernel of the canonical map $H^0(\mathcal{O}_{\P^n}(1))\to H^0(\mathcal{O}_X(1))$ and measures incompleteness of the linear system giving the embedding. Generally, for any $d\ge1$ we call $\dim_k H^1(\mathcal{I}_{X/\P^n}(d))$ $d$-normality of $X\subset\P^n$ and $X$ is said to be \ti{$d$-normal} if $H^1(\mathcal{I}_{X/\P^n}(d))=0$.
\item For a coherent sheaf $\mathcal{F}$ on a projective scheme $X$ over $k$, $h^i(X,\mathcal{F})$ means $\dim_k H^i(X,\mathcal{F})$.
\item (Betti numbers) For a graded $R$-module $M$, we define \textit{graded Betti numbers $\beta^R_{i,j}(M)$ of $M$} by $\dim_k\Tor_i^{R}(M,k)_{i+j}$. We denote it as $\beta_{i,j}(M)$ or $\beta_{i,j}$ if it is obvious. For a homogeneous ideal $I \subset R$, note that $\Tor^R_i(R/I,k)_{i+j}=\Tor^R_{i-1}(I,k)_{i-1+j+1}$. So $\beta^R_{i,j}(R/I)=\beta^R_{i-1,j+1}(I)$.
\item (Dimension and Degree) When we call the \textit{dimension of a closed subscheme $X\subset\P^n$}, denoted by $\dim X$, it means the degree of the Hilbert polynomial of $R/I_X$. The \ti{codimension of $X$} is $n-\dim X$. We also define the \ti{degree of $X$}, denoted by $\deg(X)$, as $(\dim X)!$ times the leading coefficient of this Hilbert polynomial.
\item (Arithmetic depth) When we refer the \textit{depth of $X\subset\P^n$}, denoted by $\depth_R(X)$ or simply $\depth(X)$, we mean the arithmetic depth of $X$, i.e. $\depth_{R}(R/I_X)$.
\item (Generic initial ideal) Given a homogeneous ideal $I \subset R$ and a term order $\tau$, there is a Zariski open subset
$U \subset GL_{r+1}(k)$ such that $\ini_{\tau}(g(I))$ for $g \in U$ is constant. We will call this constant $\ini_{\tau}(g(I))$
 the \ti{generic initial ideal of $I$} and denote it by $\Gin_{\tau}(I)$.  
\item (Borel fixed property) The generic initial ideal $\Gin_{\tau}(I)$ of $I$ has \ti{Borel fixed property}, which is a nice combinatorial property. In characteristic 0, we say that a monomial ideal $J$ has Borel-fixed property if $x_im \in J$ for a monomial $m$, then $x_jm \in J$ for all $j \leq i$.
\item  (Term order) Unless otherwise stated, we always assume the generic initial ideal with respect to {\it the reverse lexicographic order} (e.g. see \cite{G}) and denote it simply by $\Gin(I)$.
\item (Saturation and Quotient) For a Borel fixed monomial ideal $J\subset R$, we will simply denote $\bigcup_{k=0}^{\infty}(J:x_n^k)$ and $\ds (J,x_n)/(x_n)$ by $J|_{x_n\to 1}$ and $J|_{x_n\to 0}$ respectively.
\item If $K=(k_0,\ldots,k_n)$ then we denote by ${\bf x}^K$ the monomial 
${\bf x}^K = x_0^{k_0}\cdots x_n^{k_n}$, and denote by $|K|$ its degree $|K|=\sum_{j=0}^{n}k_j$. For monomial ${\bf x}^K$, we defines $\max({\bf x}^K)=\max\{j : k_j>0\}~.$
\item (Monomial generators set) Let $J$ be any monomial ideal of $R$. Then, we have a set of minimal \ti{monomial} generators of $J$. We write this set as $\mathcal G(J)$. For each $d\ge 0$, we also write $\mathcal G(J)_d$ for the subset of minimal monomial generators of degree $d$.
\item For a Borel fixed monomial ideal $J$, we denote the subset $\{ {\bf x}^K \in \mathcal G(J)_d \mid \max({\bf x}^K)=i \}$ of $\mathcal G(J)_d$ by $\mathcal{M}_i(d,J)$. Thus, it holds that $\mathcal G(J)_d=\bigcup_{i=0}^{n}\mathcal{M}_i(d,J)$.
\end{itemize}

In general, a great deal of fundamental information about a homogeneous ideal $I$ can be obtained from $\Gin(I)$. Now we recall some known facts concerning generic initial ideals from \cite{BS,G}, which will be used throughout the remaining parts of the paper:

\begin{Thm}{\cite{BS,G}}\label{BS}
Let $I$ be a homogeneous ideal and $L$ be a general linear form in $R=k[x_0,\ldots, x_n]$. Consider the ideal $\bar{I}=(I,L)/(L)$ as a homogeneous ideal of polynomial ring 
$S=k[x_0,\ldots,x_{n-1}]$. Then we have
\begin{itemize}
 \item[(a)] $\Gin(\bar{I})=\frac{(\Gin(I),x_n)}{(x_n)}=\Gin(I)|_{\,x_n\rightarrow0}$;
 \item[(b)] $\Gin(\bar{I}^{\sat})=\bigcup_{k=0}^{\infty}(\Gin(\bar{I}):x_{n-1}^k)=(\Gin(I)|_{\,x_n\rightarrow0})|_{x_{n\!-\!1}\rightarrow
1}$;
\item[(c)] $\sat(I)=\sat(\Gin(I))=$ the maximal degree of generators involving $x_n$;
\item[(d)] $\reg(I)=\reg(\Gin(I))=$ the maximal degree of generators of $\Gin(I)$.
\end{itemize}

\end{Thm}


From the following result we can compute the minimal free resolution of Borel fixed monomial ideals. Note that generic initial ideals have Borel fixed property (see e.g. \cite{BS,Ga}).
\begin{Thm}\label{EK} (Eliahou and Kervaire) Let $J$ be a Borel fixed monimial ideal of $R=k[x_0,\ldots,x_n]$. Denote by $\mathcal G (J)$ the set of minimal (monomial) generators of $J$ and ${\mathcal G (J)}_d$ the elements of $\mathcal G (J)$ having degree $d$. Then,
\begin{equation*}
\beta_{i,d}(R/J)=\beta_{i-1,d+1}(J)=\sum_{T\in{\mathcal G (J)}_{d+1}}{\max(T) \choose i-1}~.
\end{equation*}
\end{Thm}
\begin{proof}
See the main result in \cite{EK}.
\end{proof}

Now, we present a useful lemma for the remaining part, which also comes from the Borel fixed propery.

\begin{Lem}\label{lem:20140127-0}
Let $J\subset R=k[x_0,\ldots, x_n]$ be a Borel fixed monomial ideal and $T\in R_{d}\setminus J_d$ be a monomial of degree $d$. Suppose that $Tx_j\in J_{d+1}$ for some $j\geq\max(T)$. Then we have
\[ Tx_i \in \mathcal G(J)_{d+1} \text{ for every } i \text{ such that }\max(T)\leq i \leq j.\] 
\end{Lem}
\begin{proof}
Let us write $T$ as $x_{k_1}x_{k_2}\cdots x_{k_{d}}$ where $k_1\le k_2\le\cdots\le k_{d}=\max(T)$. Since $Tx_j\in J_{d+1}$, by Borel fixed property, we see that $Tx_i \in J_{d+1}$ \text{ for each } $i$ \text{ with } $\max(T)\leq i \leq j$. Suppose that $Tx_i$ is not a minimal generator of $J$ for some $\max(T)\leq i \leq j$. Then there is a monomial $M\in J_{d}$ such that $Tx_i=M x_\ell$ for some variable $x_\ell\in R_1$. Since $T\notin J_d$, we see that $M\neq T$ (so, $x_\ell\neq x_i$). Now we could write $M=x_{k_1}x_{k_2}\cdots \widehat{x_{\ell}}\cdots x_{k_{d}}x_i$ for some $k_1\le\ell\le k_{d}$. So, using Borel fixed property again, we see that $T=x_{k_1}x_{k_2}\cdots x_{k_{d-1}}\in J_{d}$, which contradicts to $T\notin J_d$.\end{proof}

\begin{Remk}\label{lem:20140127-1}
Lemma \ref{lem:20140127-0} can be also used to show a variant as follows:

``Let $J\subset R=k[x_0,\ldots, x_n]$ be a Borel fixed monomial ideal and $T\in R_{d}$ be a monomial. Suppose that $Tx_j\in \mathcal G(J)_{d+1}$ for some $j\geq\max(T)$. Then, we have 
\[\trm{$Tx_i \in \mathcal G(J)_{d+1} \text{ for each } i \text{ with }\max(T)\leq i \leq j$ .''}\]
\end{Remk}

Let $I$ be a homogeneous ideal of $R$. For any monomial term order $\tau$ there exists a flat family of ideals $I_t$ with $I_0=\ini_{\tau}(I)$ (the initial ideal of $I$) and $I_t$ canonically isomorphic to $I$ for all $t\neq 0$ (this implies that $\ini_{\tau}(I)$ has the same Hilbert function as that of $I$). Using this result, we have

\begin{Thm}\label{cancel}(The Cancellation Principle \cite[corollary~1.21]{G}) Choose any monomial term order $\tau$. For any homogeneous ideal $I$ and any $i$ and $d$, there is a complex of $k$-modules $V^d_{\bullet}$ such that
\begin{equation*}
V^d_i\simeq \Tor^R_i(\ini_{\tau}(I),k)_{i+d}~,~~~H_i(V^d_{\bullet})\simeq\Tor^R_i(I,k)_{i+d}~.
\end{equation*}
In particular, this implies that
$$\beta_{i,d}(R/I)\le \beta_{i,d}(R/\ini_{\tau}(I))~\textrm{for any $i, d$}.$$
\end{Thm}

\begin{Ex}
Consider a complete intersection $X$ of type $(2,2,2)$ in $\P^4$. The followings are Betti tables of $I_X$ and $\Gin(I_X)$.

\begin{center}
  \begin{tabular}{ccccccc}
$I_X$ &  & $\Gin(I_X)$\\[1ex]
{\tiny\texttt{
 \begin{tabular}{c|ccccccccccccccccccccccccccccc}
           &0 & 1 & 2 &3 &... \\[1ex]
        \hline\\
        0 & 1 & 0 &  0 & 0   & ...\\[1ex]
        1 & 0 & 3 &  0 & 0   & ...\\[1ex]
        2 & 0 & 0 &  3 & 0   & ...\\[1ex]
        3 & 0 & 0 &  0 & 1   & ...\\[1ex]
 \end{tabular}
}}

& $\Longleftrightarrow$
& {\tiny\texttt{
 \begin{tabular}{c|ccccccccccccccccccccccccccccc}
           &0 & 1 & 2 &3 &... \\[1ex]
        \hline\\
        0 & 1 & 0 &  0 & 0   & ...\\[1ex]
        1 & 0 & 3 &  2 & 0   & ...\\[1ex]
        2 & 0 & 2 &  4 & 2   & ...\\[1ex]
        3 & 0 & 1 &  2 & 1   & ...\\[1ex]
 \end{tabular}
}}\\[2ex]
\end{tabular}
\end{center}
The cancellation principle says that the minimal free resolution of $I_X$ is obtained from that of $\Gin(I_X)$ by canceling some adjacent terms of the same shift.
\end{Ex}

\begin{Remk}\label{remk:20140116-1}
 From Theorem~\ref{cancel}, we see that $\beta_{i,d}(R/I)\le \beta_{i,d}(R/\Gin(I))$ for any $i, d$. We can ask when the equality holds.  One case that we want to describe is the following. If we have 
 $$\beta_{k-1,m+1}(R/\Gin(I))=\beta_{k+1,m-1}(R/\Gin(I))=0, \text{ for some $k, m$ }$$
 then $\beta_{k,m}(R/I)=\beta_{k,m}(R/\Gin(I))$ by the cancelation principle. For example, if $\Gin(I)$ has a linear resolution then $I$ also has the same linear resolution as that of $\Gin(I)$.
\end{Remk}
Another case is the following, which will be used in this paper.
\begin{Lem}[Marginal Betti number]\label{lem:20140116-1}
For a homogeneous ideal $I \subset R$, suppose that $I$ is $(d+1)$-regular and the projective dimension of $R/I$ is at most $p$. Then we have  
  $$\beta_{p,d}(R/I)=\beta_{p,d}(R/\Gin(I)).$$
\end{Lem}
\begin{proof}
Note that the projective dimension of $R/I$ is the same as that of $R/\Gin(I)$ (see e.g. \cite[corollary 2.8]{AM}). So, $\beta_{p+1,d-1}(R/\Gin(I))=0$. Since $I$ is $(d+1)$-regular, from the result of Bayer-Stillman (Theorem \ref{BS} (d)), $\Gin(I)$ is also $(d+1)$-regular. Hence we have $\beta_{p-1,d+1}(R/\Gin(I))=0$. Then, it follows from 
Remark~\ref{remk:20140116-1} that $\beta=\beta_{p,d}(R/I)=\beta_{p,d}(R/\Gin(I))$.
\begin{figure}[!htb]
{\tiny\texttt{
 \begin{tabular}{c|ccccccccccccccccccccccccccccc}
             & 0 & 1 & 2 & $\cdots$  & p-1 &p& p+1&$\cdots$\\[1ex]
        \hline\\
        0 & 1 & 0 & 0  & $\cdots$ & 0& 0 & 0&$\cdots$\\[1ex]
        1 & 0 & * & *  & $\cdots$ & *& * & 0&$\cdots$ \\[1ex]
                $\vdots$ & $\vdots$ & $\vdots$ & $\vdots$  & $\cdots$ & $\vdots$& $\vdots$ & $\vdots$&$\cdots$ \\[1ex]  
 d-1   & 0 & * & *  & $\cdots$ & *& * & 0&$\cdots$ \\[1ex]  
 d   & 0 & * & *  & $\cdots$ & *& $\beta$  & 0&$\cdots$ \\[1ex]  
 d+1 &0 & 0 & 0  & $\cdots$ & 0& 0 & 0&$\cdots$ \\[1ex]       
         \end{tabular}}}
\caption{We call $\beta$ the \ti{marginal Betti number of $R/I$}.}
\label{marginal_betti}
\end{figure}
\end{proof}

\section{$\ND$ property and structures of reverse lexicographic Gins of 3-regular subschemes}\label{sect_str}

We begin with the following result.

\begin{Prop}[Basic inequality for degrees of $\ND$ subschemes]\label{basic_ineq}
Let $X$ be a closed subscheme of codimension $e$ in $\P^n$ satisfying $\ND$ property. Then, the following holds
\[\deg(X)\ge e+1~.\]
\end{Prop}

\begin{proof}
Let us prove by induction on $\dim X$. First, if $\dim X=0$ (so, $e=n$), then by $\ND$ property of $X$, $H^0(\mathcal I_{X/\P^n}(1))=0$ and the sequence $0\to H^0(\mathcal O_{\P^n}(1))\to H^0(\mathcal O_X(1))$ tells us that $\deg(X)\ge n+1=e+1$, as we wish. Now, suppose that the statement holds for any $\ND$ subscheme of dimension $\le k$ for some $k\ge 0$. Then, let us consider a $\ND$ subscheme $X$ of dimension $k+1$. For a general hyperplane section $X\cap H$ still satisfies $\ND$ property and has the same codimension and degree as those of $X$, by induction hypothesis we have
\begin{align*}
\deg(X)=\deg(X\cap H)\ge e+1~,
\end{align*}
and the assertion is proved.
\end{proof}

Note that this is not true in general (see Example \ref{non-ND_ex}).

\begin{Thm}[$K_{p,1}$-theorem for $\ND$ subscheme]\label{thm:20140212-1}
Let $X$ be a closed subscheme of codimension $e$ in $\P^n$ satisfying $\ND$ property. Then we have 
$$\beta_{i,1}(R/\Gin(I_X))=0~\trm{for any $i>e=\codim(X,\P^{n})$}.$$ Therefore, we also have $\beta_{i,1}(R/I_X)=0$ for any $i>e$.
\end{Thm}

\begin{proof}
The cancellation principle (Theorem~\ref{cancel}) implies that
$$\beta_{i,1}(R/I_X)\le \beta_{i,1}(R/\Gin(I_X))~\textrm{for any $i$}~.$$
Therefore, it suffices to show that
$$ \beta_{i,1}(R/\Gin(I_X))=0~\textrm{for any $i>e$}~.$$
Note that the assumption on $\ND$ property of $X$ guarantees that the linear section $X\cap \Lambda$ is nondegenerate for a general linear subspace $\Lambda$ of dimension $\geq e$.

Suppose that $\beta_{i,1}(R/I_X)\neq0$ for some $i>e$. Then, from Theorem~\ref{EK} we have
\begin{eqnarray*}
0&<&  \beta_{i,1}(R/\Gin(I_X))~=~\sum_{T\in \mathcal G(\Gin(I_X))_2}\binom{\max(T)}{i-1}~,\\
&\Rightarrow&\exists~ T_0\in \mathcal G(\Gin(I_X))_2 \textrm{ such that } \max(T_0)\geq i-1\geq e.
\end{eqnarray*}
By Borel fixed property we may assume that $T_0^\prime=x_1 x_{i-1}$ in $\Gin(I_X)
_2$.
Consider a general linear section $X\cap\Lambda$ with $\dim \Lambda=i-1$. Then, by Theorem~\ref{BS} (a) and (b),
\[\Gin(I_{X\cap\Lambda/\Lambda})=\left[\frac{(\Gin(I_X), x_{i}, x_{i+1},\ldots, x_{n})}{(x_{i}, x_{i+1},\ldots, x_{n})}\right]^{\sat}=\left[\frac{(\Gin(I_X), x_{i}, x_{i+1},\ldots, x_{n})}{(x_{i}, x_{i+1},\ldots, x_{n})}\right]_{x_{i-1}\to 1},\]
and thus $T_0^\prime$ gives $x_1\in \Gin(I_{X\cap \Lambda/\Lambda})$, which implies that $I_{X\cap \Lambda/\Lambda}$ has a linear form. This contradicts that $X\cap \Lambda$ is nondegenerate.
\end{proof}

\begin{Remk}\label{remk_NP}
This was first known by Green's $K_{p,1}$-theorem for compact complex manifolds (\cite{G2}) and later by Nagel-Pitteloud for a more general setting (\cite{NP}). Here, by utilizing $\ND$ assumption, the proof of Theorem \ref{thm:20140212-1} is quite simple. But, it says the result for more general cases than those treated in \cite{G2,NP}. For example, a case such as (c) in Example \ref{ND_ex}, a subscheme whose ideal contains some power of a linear form, can not be covered by previous results, while this method can do.
\end{Remk}

\begin{Ex}[Non-$\ND$ subschemes]\label{non-ND_ex}
If a closed subscheme $X\subset \P^n$ does not satisfy $\ND$, then Theorem \ref{thm:20140212-1} is no longer true. For examples,
\begin{itemize}
\item[(a)] (Two planes in $\P^4$) Let $X\subset \P^4$ be a closed subscheme  defined by 
$$I_X=(x_0,x_1)\cap (x_2,x_3)\subset R=k[x_0,x_1,x_2,x_3,x_4].$$
Since $X$ is a union of two planes in four space meeting in a single point, we see that $X$ and $X\cap H$ are nondegenerate for a general hyperplane $H$. However, because $X\cap H$ is a union of two skew lines in $H$, for a general 2-plane $\Lambda$ the intersection $X\cap \Lambda$ is a set of two points, degenerate. Thus, $X$ does not satisfy $\ND$ property. We see that $\deg(X)=2$ and $2\ngtr2+1$, so the inequality in Proposition \ref{basic_ineq} does not hold for this non-$\ND$ subscheme. The Betti table of $X$ is given by
\[{\tiny\texttt{
 \begin{tabular}{c|ccccccccccccccccccccccccccccc}
             & 0 & 1 & 2 & 3  & 4 & \\[1ex]
        \hline\\
        0 & 1 & 0 & 0  & 0 & 0& \\[1ex]
        1 & 0 & 4 & 4  & 1 & 0& \\[1ex]
        2 & 0 & 0 & 0  & 0 & 0& \\[1ex]
         \end{tabular}
}}~.\]
Note that codimension of $X$ is two but $\beta_{3,1}\neq 0$, that is, Theorem \ref{thm:20140212-1}  does not hold.
\item[(b)] Consider the following non-reduced closed subscheme $X\in \P^3$ defined by the ideal
\[
\begin{array}{llllllll}
 I_X&=&(x_0x_2-x_1x_2 -x_0x_1x_2+x_1^2x_2 -x_2^2 -x_0^2x_1+x_0x_1^2 -x_1x_2 -x_0x_2)\\
  &= &(x_0,x_2)\cap(x_1,x_2)\cap(x_0-x_1,x_2)\cap(x_0,x_1,x_2)^2.
 \end{array}
 \]
 Then $X$ is a union of three lines in a plane $x_2=0$ with an embedded point at the origin. 
Since a general hyperplane section gives $3$-collinear points in a line, we see that $X$ does not satisfy $\ND$. On the other hand, Betti table of $R/I_X$ is given by
 \[{\tiny\texttt{
 \begin{tabular}{c|ccccccccccccccccccccccccccccc}
             & 0 & 1 & 2 & 3  & 4 & ... \\[1ex]
        \hline\\
        0 & 1 & 0 & 0  & 0 & 0&... \\[1ex]
        1 & 0 & 3 & 3  & 1 & 0&... \\[1ex]
        2 & 0 & 1 & 1  & 0 & 0&... \\[1ex]
         \end{tabular}
}}\]
such that $\beta_{e+1,1}(R/I_X)=1\neq 0$, $\beta_{e+1,2}(R/I_X)=0$ and $\beta_{e,2}(R/I_X)=1$. In this case, Theorem~\ref{main_thm} (a) is not true. But, by simple computation, we also see that $h^1(\mathcal I_{X/\P^3}(1))=0$ and $h^1(\mathcal I_{X\cap H/H}(1))=1$. So, even though  $X$ is non-$\ND$, the relation in Theorem~\ref{main_thm} (b) still holds.
\end{itemize}
\end{Ex}

\begin{Remk} We suspect that it might not be possible that the cancellation (see Theorem \ref{cancel}) happens at the Betti numbers whose homological index is greater than the codimension $e$. If this is true, then, as shown in Example \ref{non-ND_ex} (b), the tailing Betti formulae in Theorem~\ref{main_thm} (b) is still true without $\ND$ assupmtion.
\end{Remk}

\begin{Coro}\label{coro:20140219}
Let $X$ be a closed scheme of codimension $e$ in $\P^n$ satisfying $\ND$ property. Suppose that $X$ is $3$-regular. Then we have 
$$\beta_{i,2}(R/I_X)=\beta_{i,2}(R/\Gin(I_X))~\trm{for any $i\geq e$.}$$
\end{Coro}
\begin{proof}
For each $i\geq e$, we see from Theorem~\ref{thm:20140212-1} that $\beta_{i+1,1}(R/\Gin(I_X))=0$. Since $X$ is $3$-regular  we also have $\beta_{i-1,3}(R/\Gin(I_X))=0$ for any $i$. Then it follows from Remark~\ref{remk:20140116-1} that  $\beta_{i,2}(R/I_X)$ is equal to $\beta_{i,2}(R/\Gin(I_X))$, as we wished.
\end{proof}

Once $X$ is $(d+1)$-regular, $X$ is $m$-normal for all $m\ge d$ (see e.g. \cite[pg. 100]{Mum}). Then, what can we say about $h^1(\mathcal{I}_{X/\P^n}(d-1))$? It is quite remarkable that the next normality $h^1(\mathcal{I}_{X/\P^n}(d-1))$ can be read off from the \ti{generation structure} of $\Gin(I_X)$, as the following proposition says:

\begin{Prop}\label{lem:20140113-1}
Let $X$ be any closed subscheme in $\P^n$ and $I_X$ be the saturated defining ideal of $X$. Suppose that $X$ is (d+1)-regular for some $d\geq 1$. Then, we have
\begin{itemize}
\item[(a)] $h^1(\mathcal{I}_{X/\P^n}(d-1))=\mid \mathcal{M}_{n-1}(d+1,\Gin(I_X))\mid$;
\item[(b)] (Marginal Betti number) $\beta_{n,d}(R/I_X)=h^1(\mathcal I_{X/\P^n}(d-1))$.
\end{itemize}
\end{Prop}

\begin{proof} (a) Let $H$ be the hyperplane defined by a general linear form $L$, $\bar{I}=\frac{(I_{X},L)}{(L)}$ and 
$$\mathcal{M}_{n-1}(d+1,\Gin(I_X))=\{T_1x_{n-1}, T_2x_{n-1}, \ldots, T_rx_{n-1}\}.$$ 
Since $X$ is $(d+1)$-regular, $X$ is also $d$-normal, that is $H^1(\mathcal I_{X/\P^n}(d))=0$. Then, from the exact sequence
\[
0\rightarrow H^0(\mathcal I_{X/\P^n} (d-1)) \stackrel{\cdot L}{\longrightarrow} H^0(\mathcal I_{X/\P^n} (d)) \rightarrow
H^0(\mathcal I_{X\cap H/H}(d)) \rightarrow H^1(\mathcal I_{X/\P^n} (d-1))\rightarrow 0~,\]
it follows that
\begin{align*}
h^1(\mathcal{I}_{X/\P^n}(d-1))=&\,\dim_k(\bar{I}^{sat}/\bar{I})_d=\,\dim_k(\Gin(\bar{I}^{sat})/\Gin(\bar{I}))_d\\[1ex]
=&\,\dim_k\left[\frac{(\Gin(I_X)|_{\,x_n\rightarrow0})|_{x_{n\!-\!1}\rightarrow
1}}{\Gin(I_X)|_{x_n\rightarrow0}}\right]_d\quad\trm{(by Theorem \ref{BS} (a),(b)) $\cdots(\ast)$}              
\end{align*}
Now, we claim that $(\ast)$ is equal to $\mid \mathcal{M}_{n-1}(d+1,\Gin(I_X))\mid$.

It is clear that each element of $\mathcal{M}_{n-1}(d+1,\Gin(I_X))$ contributes to $(\ast)$ by one. For the converse, choose any degree $d$ monomial $T\in \left[(\Gin(I_X)|_{\,x_n\rightarrow0})|_{x_{n\!-\!1}\rightarrow
1}\setminus\Gin(I_X)|_{x_n\rightarrow0}\right]$, so that $\max(T)<n-1$, $T\notin\Gin(I_X)_d$ and $T x_{n-1}^k\in\Gin(I_X)_{d+k}$ for some $k\ge 1$. Let $k_0$ be the minimum of such $k$'s. Then, we are enough to show that $k_0=1$, for this also implies $T x_{n-1}\in \mathcal{M}_{n-1}(d+1,\Gin(I_X))$ by Lemma \ref{lem:20140127-0}.

Suppose that $k_0\ge 2$. By the result of Bayer-Stillman (Theorem \ref{BS} (d)), $\Gin(I_X)$ is also $(d+1)$-regular and we could write $Tx_{n-1}^{k_0}$ as $M x_\ell$ for some $M\in \Gin(I_X)_{d+k_0-1}$ and some $x_\ell\in R_1$. Due to minimality of $k_0$, note that $x_\ell\neq x_{n-1}$ (i.e. $\ell<n-1$). Then, $\left(T/x_{\ell}\right)\cdot x_{n-1}^{k_0}\in\Gin(I_X)$ and by Borel fixed property $T\cdot x_{n-1}^{k_0-1}=\left(T/x_{\ell}\right)\cdot x_{\ell}x_{n-1}^{k_0-1}$ also belong to $\Gin(I_X)$, which is again a contradiction to minimality of $k_0$.

(b) Since $I_X$ is a saturated ideal, $\depth(R/I_X)\geq 1$, which implies that the projective dimension of $R/I_X$ is at most $n$. Hence, we see that
\begin{eqnarray*}
\beta_{n,d}(R/I_X)&=&\beta_{n,d}(R/\Gin(I_X))\quad(\textrm{by Lemma~\ref{lem:20140116-1}})\\
&=& \sum_{T\in \mathcal G(\Gin(I_X))_{d+1}}{\max(T) \choose n-1}\quad(\textrm{by Theorem~\ref{EK}})\\
&=&\#\{{\bf x}^K\in \mathcal G(\Gin(I_X))_{d+1}~|~\max({\bf x}^K)=n-1\}\\
&=& \mid \mathcal{M}_{n-1}(d+1,\Gin(I_X)) \mid ~= h^1(\mathcal I_{X/\P^n}(d-1))\quad(\textrm{by (a)}),
\end{eqnarray*}
as we wished.
\end{proof}

\begin{Remk}\label{remk:20140219}
\begin{itemize}
\item[(a)] With the same notation as Proposition \ref{lem:20140113-1}, the argument used in the proof of Proposition \ref{lem:20140113-1} (a) actually shows that
\begin{align*}
\trm{$\left(\Gin(\bar{I}^{sat})/\Gin(\bar{I})\right)_d$ is exactly same as the $k$-vector space $\langle T_1, T_2, \ldots, T_r\rangle$,}
\end{align*}
where $\mathcal{M}_{n-1}(d+1,\Gin(I_X))=\{T_1 x_{n-1}, T_2 x_{n-1}, \ldots, T_r x_{n-1}\}$.
 \item[(b)] If we only concern about Proposition \ref{lem:20140113-1} (b), then it can be also shown using \ti{local duality} (see e.g. \cite[Appendix A]{E2} for the statement of local duality). Here, the proof is given by a completely different method using combinatorial properties of Gins.
\end{itemize}
\end{Remk}

\begin{Ex}[Quadratic normality of a rational curve]\label{ex:20140112-3}
Consider a smooth rational curve $C\subset \mathbb P^3$ of degree
5 defined by the map:
\[[s,t] \rightarrow [s^5,s^4t,st^4,t^5].\]
Then, using \texttt{Macaulay 2} (\cite{M2}) we know that $I_C$ and $\Gin(I_C)$ are minimally generated as:
\[
\begin{array}{llllllllllllllllllll}
\bullet & I_C = (x_1x_2-x_0x_3,x_2^4-x_1x_3^3,x_0x_2^3-x_1^2x_3^2,
    x_0^2x_2^2-x_1^3x_3,x_1^4-x_0^3x_2)\\[1ex]
\bullet & \Gin(I_C)
=(x_0^2,x_0x_1^3,x_1^4,x_0x_1^2x_2,x_1^3x_2)
\end{array}
\]
and the Betti table of $R/I_C$ is given by
\[{\tiny\texttt{
 \begin{tabular}{c|ccccccccccccccccccccccccccccc}
           & 0 & 1 &  2 & 3   & 4\\[1ex]
        \hline\\
        0 & 1 & 0 &  0 & 0   & 0\\[1ex]
        1 & 0 & 1 &  0 & 0   & 0\\[1ex]
        2 & 0 & 0 &  0 & 0   & 0\\[1ex]
        3 & 0 & 4 &  6 & 2   & 0\\[1ex]
 \end{tabular}
}}~.\]
Note that $\mathcal{M}_{2}(4,I_C)=\{x_1^3x_2, x_0x_1^2x_2\}$. Since the maximal degree of generators of $\Gin(I_C)$ is $4$, we know that $I_C$ is $4$-regular (Thoerem~\ref{BS} (d)), as we see in the table above. Then, by Proposition \ref{lem:20140113-1} we compute the 2-normality of $C\subset \mathbb P^3$
$$h^1(I_{C/\P^3}(2))=~\mid \mathcal{M}_{2}(4,I_C)\mid~=2~,$$
which also coincides with the marginal Betti number $\beta_{3,3}(R/I_C)$.
\end{Ex}

The following proposition is a crucial part of obtaining formulae in the main Theorem \ref{main}, which describes a peculiar aspect of the generation structure of the reverse lexicographic Gins of 3-regular subschemes. We prove this by exploiting Borel fixed property of Gins and our key definition, $\ND$ property.

\begin{Prop}\label{prop:20140211}
Let $X$ be a closed subscheme of codimension $e$ in $\P^n$ satisfying $\ND$ property. If $H$ is a general hyperplane in $\mathbb P^n$ and 
\[ \mathcal{M}_{n-1}(3,\Gin(I_X))=\{T_1x_{n-1}, T_2x_{n-1},\ldots, T_rx_{n-1}\}~,\]
then we have (here, $\bigsqcup$ means ``disjoint union''):
\begin{itemize}
\item[(a)] For any $T_k x_{n-1}\in\mathcal{M}_{n-1}(3,\Gin(I_X))$, $\max(T_k)\le e-1$; 
\item[(b)] For each $e-1\leq i\leq n-1$, $T_k x_i\in \mathcal{M}_{i}(3,\Gin(I_X))$.
\end{itemize}
Now, further suppose that $X$ is $3$-regular. Then, as a set of monomials we have the followings:
\begin{itemize}
\item[(c)] (degree 2 level) $\mathcal G(\Gin(I_{X\cap H/H}))_2=\mathcal G(\Gin(I_X))_2 ~\bigsqcup~ \{T_1, T_2,\ldots,T_r \}$;
\item[(d)] (degree 3 level) $\mathcal G(\Gin(I_{X\cap H/H}))_3\subset \mathcal G(\Gin(I_X))_3$; 
\item[(e)] (degree 3 level, continued) Moreover, for every $e-1\leq i \leq n-1$, we have $$\mathcal{M}_i(3,\Gin(I_X))=\mathcal{M}_i(3,\Gin(I_{X\cap H/H}))~\bigsqcup~\{T_1x_i, T_2x_i,\ldots,T_rx_i\}~.$$  In particular, we obtain
\[\mid \mathcal{M}_i(3,\Gin(I_X))\mid=\mid \mathcal{M}_i(3,\Gin(I_{X\cap H/H}))\mid +h^1(\mathcal I_{X/\P^n}(1))~.\]
\end{itemize}
\end{Prop}

\begin{proof}
(a) Suppose that $\max(T_k)=j\geq  e$ for some $j$. Then by Borel fixed property, we may assume that $\Gin(I_X)$ contains a monomial $x_1x_j^2$. For a general linear space $\Lambda$ of dimension $j$, it follows from Theorem~\ref{BS} (a) and (b),
\[\Gin(I_{X\cap\Lambda/\Lambda})=\left[\frac{(\Gin(I_X), x_{j+1},\ldots, x_{n})}{(x_{j+1},\ldots, x_{n})}\right]^{\sat}=\left[\frac{(\Gin(I_X), x_{j+1},\ldots, x_{n})}{(x_{j+1},\ldots, x_{n})}\right]_{x_{j}\to 1},\]
and thus $x_1\in \Gin(I_{X\cap \Lambda/\Lambda})$, which implies that $I_{X\cap \Lambda/\Lambda}$ has a linear form. This contradicts that $X$ satisfies $\ND$ property. 

\noindent (b) From Remark~\ref{lem:20140127-1}, we can say $T_k x_i\in \mathcal G(\Gin(I_X))_3$ (i.e. $T_k x_i\in \mathcal M_i(3,\Gin(I_X))$) for each $i\ge e-1$. 

\noindent (c) Since $X$ is nondegenerate, we see that $\Gin(I_X)_1=0$. Hence every monomial in $\Gin(I_X)_2$ is a minimal generator of degree $2$ (i.e. an element of $\mathcal G(\Gin(I_X))_2$). Furthermore, from the $\ND$ property we see $\Gin(I_{X\cap H/H})_1=0$, which means that every monomial in $\Gin(I_{X\cap H/H})_2$ also belongs to $\mathcal G(\Gin(I_{X\cap H/H}))_2$. Since there exists an exact sequence as follows
$$ \Gin(I_X)_1=0 \to \Gin(I_X)_2 \to \Gin(I_{X\cap H/H})_2 \to H^1(\P^n,\widetilde{\Gin(I_X)}(1))~,$$
where $\widetilde{\Gin(I_X)}$ is the sheafification of $\Gin(I_X)$, it implies that $\mathcal G(\Gin(I_X))_2\subset \mathcal G(\Gin(I_{X\cap H/H}))_2$ and the difference set $\mathcal G(\Gin(I_{X\cap H/H}))_2\setminus\mathcal G(\Gin(I_X))_2$ consists of all the monomials in 
$$\left(\frac{\Gin(I_{X\cap H/H})}{\Gin(I_X)}\right)_2=\left(\frac{(\Gin(I_X)|_{\,x_n\rightarrow0})|_{x_{n\!-\!1}\rightarrow
1}}{\Gin(I_X)|_{x_n\rightarrow0}}\right)_2=\left(\frac{\Gin({\bar{I}_X}^{sat})}{\Gin(\bar{I}_X)}\right)_2~,$$ 
where $\Gin(I_X)_2=(\Gin(I_X)|_{x_n\rightarrow0})_2$ ($\because$ by Thoerem~\ref{BS} (c), $\max({\bf x}^K)\leq n-1$ for each ${\bf x}^K\in \mathcal G(\Gin(I))$) and $\bar{I}_X$ means the ideal $\frac{I_X+(H)}{(H)}$ by abuse of notation. Now that $X$ is $3$-regular, we know that this is equal to $\{T_1, T_2,\ldots,T_r \}$ by Remark~\ref{remk:20140219} (a). Thus, as a set we establish
$$\mathcal G(\Gin(I_{X\cap H/H}))_2=\mathcal G(\Gin(I_X))_2~\bigsqcup~ \{T_1, T_2,\ldots,T_r \}.$$

\noindent (d) Because $X$ is $3$-regular (so, $\Gin(I_X)$ is also $3$-regular), the following exact sequence 
\[
\begin{array}{ccccccccccccccccccccccc}
&\Gin(I_X)_2 &\stackrel{\times x_n}{\to} &\Gin(I_X)_3 &{\to} & \Gin(I_{X\cap H,H})_3 &\to 0&=H^1(\mathbb P^n, \widetilde{\Gin(I_X)}(2))& \\[1ex]
& \quad\ds\parallel  && \quad\ds\parallel && \quad\ds\parallel  &  \\[1ex]
0 \to &\Gin(I_X)_2 &\stackrel{\times x_n}{\to} &\Gin(I_X)_3 &{\to} & \ds\left[\frac{(\Gin(I_X),x_n)}{(x_n)}\right]_3 &\to 0&  & \\[1ex]
\end{array}
\]
says that every monomial in $\Gin(I_{X\cap H,H})_3$ can be lifted to that in $\Gin(I_X)_3$.  

Assume that ${\bf x}^K\in \mathcal G(\Gin(I_{X\cap H,H}))_3$. We see at least that ${\bf x}^K$ can be considered as a monomial in $\mathcal \Gin(I_{X})_3$. If ${\bf x}^K$ is not a minimal generator of $\Gin(I_X)$, then there is a monomial $T\in \mathcal G(\Gin(I_X))_2$ such that $T$ divides ${\bf x}^K$. But, by (c) we see that $T$ is also in $\mathcal G(\Gin(I_{X\cap H,H}))_2$. This is a contradiction to the assumption that ${\bf x}^K\in \mathcal G(\Gin(I_{X\cap H,H}))_3$. Hence we conclude that $\mathcal G(\Gin(I_{X\cap H,H}))_3\subset \mathcal G(\Gin(I_X))_3$.

\noindent (e) For any $e-1\le i\le n-1$, we first know that $\{T_1x_i, T_2x_i,\ldots,T_rx_i\}\subset\mathcal{M}_i(3,\Gin(I_X))$ by (b). Combined with the result (d), we have
\[\mathcal{M}_i(3,\Gin(I_{X\cap H/H}))\bigsqcup\{T_1x_i, T_2x_i,\ldots,T_rx_i\}\subset \mathcal{M}_i(3,\Gin(I_X)).\]
Now, for the converse, choose any ${\bf x}^K\in \mathcal{M}_i(3,\Gin(I_X))$. Because $I_X$ is a saturated ideal, $\max({\bf x}^K)\leq n-1$ so that ${\bf x}^K$ can be regarded as a nonzero monomial in $((\Gin(I_X)|_{\,x_n\rightarrow0})|_{x_{n\!-\!1}\rightarrow
1})_3=\Gin(I_{X\cap H/H})_3$. If ${\bf x}^K\notin \mathcal{M}_i(3,\Gin(I_{X\cap H/H}))$, then there is a monomial $U\in \mathcal G(\Gin(I_{X\cap H/H}))_2$ such that $U$ divides ${\bf x}^K$. Since by (c) we have 
$$\mathcal G(\Gin(I_{X\cap H/H}))_2=\mathcal G(\Gin(I_X))_2 ~\bigsqcup~ \{T_1, T_2,\ldots,T_r \},$$ 
$U$ should be  contained in either $\mathcal G(\Gin(I_X))_2$ or the set $\{T_1, T_2,\ldots,T_r \}$. However, the assumption that ${\bf x}^K\in \mathcal{M}_i(3,\Gin(I_X))$ implies that $U\notin \mathcal G(\Gin(I_X))_2$. Hence $U=T_k $ for some $k$, and this conclude that ${\bf x}^K=T_k x_i \in \mathcal{M}_{i}(3,\Gin(I_X))$. Therefore, we prove that 
$$\mathcal{M}_i(3,\Gin(I_X))=\mathcal{M}_i(3,\Gin(I_{X\cap H/H}))~\bigsqcup~\{T_1 x_i, T_2 x_i,\ldots,T_r x_i\}$$ for each $e-1\leq i\leq n-1$. Further, it follows from Proposition~\ref{lem:20140113-1} (a) that  
\[
\begin{array}{llllllllll}
\mid \mathcal{M}_i(3,\Gin(I_X))\mid &=&\mid \mathcal{M}_i(3,\Gin(I_{X\cap H/H}))\mid + \mid \mathcal{M}_{n-1}(3,\Gin(I_X))\mid \\[1ex]
                                       &=&\mid \mathcal{M}_i(3,\Gin(I_{X\cap H/H}))\mid + h^1(\mathcal I_{X/\P^n}(1)),
\end{array}
\]
as we wished.
\end{proof}

\section{1-normality of general linear sections and tailing Betti numbers}\label{sect_main}

Now we are ready to prove our main result.

\begin{Thm}[Formulae for tailing Betti numbers]\label{main}
Let $X$ be a closed subscheme of codimension $e$ in $\P^n$ satisfying $\ND$ property. Suppose that $X$ is $3$-regular. Then, for any $i\geq e$ we have 
\begin{equation}\label{main_for}
\beta_{i,2}(R/I_X)=\sum_{\alpha=i}^{n} \binom{\alpha}{i}\,h^1(\mathcal I_{X\cap \Lambda^{\alpha}/\Lambda^{\alpha}}(1))~,
\end{equation}
where $\Lambda^{\alpha}$ is a general linear space of dimension $\alpha$.
\end{Thm}
\begin{proof}
Let $R=k[x_0,\cdots,x_n]$ and $S=R/H$ for a general hyperplane $H$. We will give a proof by induction on $\dim(X)\geq 0$. Suppose that $\dim(X)=0$. In this case, we only have to consider $i=n$. Hence we have a tautological identity
\[\sum_{\alpha=i}^{n} \binom{\alpha}{i}\,h^1(\mathcal I_{X\cap \Lambda^{\alpha}/\Lambda^{\alpha}}(1))=h^1(\mathcal I_{X/\P^n}(1))\]
and the claim follows from Proposition~\ref{lem:20140113-1} (b) which says that $\beta_{n, 2}(R/I_X)=h^1(\mathcal I_{X/\P^n}(1))$.  

Now assume that $\dim(X)\geq 1$. Note that $X\cap H$ is also 3-regular, has the same codimension, and satisfies $\ND$ property in the projective space $H\cong \P^{n-1}$. Then, for any $i\ge e$ we have 
\[
\begin{array}{lllllllllllllllll}
\beta_{i,2}(R/I_X)&=& \ds\beta_{i,2}(R/\Gin(I_X)) \quad \text{(by Corollary~\ref{coro:20140219})}\\[1ex]
                              &=& \ds \sum_{T\in{\mathcal G (\Gin(I_X))}_{3}}{\max(T) \choose i-1} \quad \text{(by Theorem~\ref{EK})}\\[1ex]
                              &=& \ds \sum_{k=i-1}^{n-1}{k \choose i-1}\mid \mathcal{M}_k(3,\Gin(I_X))\mid \quad(\because \max(T)\leq n-1 \text{ for all }\,T\in \mathcal G(\Gin(I_X))_3)\\[1ex]
                              &=& \ds \sum_{k=i-1}^{n-1}{k \choose i-1}\Big (\mid \mathcal{M}_k(3,\Gin(I_{X\cap H,H}))\mid +h^1(\mathcal I_{X/\P^n}(1))\Big) \quad \text{(by Proposition~\ref{prop:20140211} (e))}\\[1ex]
                              &=& \beta_{i,2}(S/\Gin(I_{X\cap H/H}))+ h^1(\mathcal I_{X/\P^n}(1))\ds \sum_{k=i-1}^{n-1}{k \choose i-1} \\[1ex]
                              &=& \ds \beta_{i,2}(S/I_{X\cap H/H})+ {n \choose i}h^1(\mathcal I_{X/\P^n}(1)) \quad \text{($\because$ $X\cap H$ is also 3-regular with $\ND$)}\\[1ex]
                              &=& \ds \sum_{\alpha=i}^{n-1}\binom{\alpha}{i}\,h^1(\mathcal I_{X\cap \Lambda^{\alpha}/\Lambda^{\alpha}}(1))+ {n \choose i}h^1(\mathcal I_{X/\P^n}(1)) \quad (\text{by inductive hypothesis})\\[2ex]
                              &=& \ds \sum_{\alpha=i}^{n} \binom{\alpha}{i}\,h^1(\mathcal I_{X\cap \Lambda^{\alpha}/\Lambda^{\alpha}}(1)), \\[1ex]
\end{array}
\]
as we wished.
\end{proof}

Before considering some applications of Theorem \ref{main}, we would like to mention that Formulae (\ref{main_for}) can be restated in a more visible way as follows:

\begin{Coro}\label{main_for_cor}
Let $X$ be any closed subscheme of codimension $e$ in $\P^n$ satisfying $\ND$ property. Suppose that $X$ is $3$-regular. Let $\mathbf{b}$ and $\mathbf{h}$ be two row vectors consisting of consecutive tailing Betti numbers and 1-normalities of general linear sections of $X$ as below:
\begin{align*}
&\mathbf{b}:=[\beta_{e,2}(R/I_X),\beta_{e+1,2}(R/I_X),\cdots,\beta_{n-1,2}(R/I_X),\beta_{n,2}(R/I_X)]\\
&\mathbf{h}:=[h^1(\mathcal{I}_{X\cap \Lambda^{e}/\Lambda^{e}}(1)),h^1(\mathcal{I}_{X\cap \Lambda^{e+1}/\Lambda^{e+1}}(1)),\cdots,h^1(\mathcal{I}_{X\cap \Lambda^{n-1}/\Lambda^{n-1}}(1)),h^1(\mathcal{I}_{X/\P^n}(1))]~.
\end{align*}
Then, we have
\begin{equation*}
\trm{$\mathbf{b}^{T}=\Xi(n,e)\cdot\mathbf{h}^{T}$~,}
\end{equation*}
where $\mathbf{b}^{T}$ (resp. $\mathbf{h}^{T}$) is the transpose of $\mathbf{b}^{T}$ (resp. of $\mathbf{h}^{T}$) and $\Xi(n,e)$ is the invertible $(n-e+1)\times(n-e+1)$-matrix such as
{\small\begin{align}\label{matXi_0}
\Xi(n,e)=&\left[\begin{array}{*6{c}}
\displaystyle{e\choose e}&\displaystyle{e+1\choose e}&\displaystyle{e+2\choose e}&\cdots&\displaystyle{n-1\choose e}&\displaystyle{n\choose e}\\
0&\displaystyle{e+1\choose e+1}&\displaystyle{e+2\choose e+1}&\cdots&\displaystyle{n-1\choose e+1}&\displaystyle{n\choose e+1}\\
0&0&\displaystyle{e+2\choose e+2}&\cdots&\displaystyle{n-1\choose e+2}&\displaystyle{n\choose e+2}\\
\vdots&\vdots&\vdots&\ddots&\vdots&\vdots\\
0&0&0&\cdots&\displaystyle{n-1\choose n-1}&\displaystyle{n\choose n-1}\\
0&0&0&\cdots&0&\displaystyle{n\choose n}
\end{array}\right]~.
\end{align}}
\end{Coro}

Now that $\det\Xi(n,e)=1$, there exists the inverse matrix $\Xi(n,e)^{-1}$. In fact, the inverse $\Xi(n,e)^{-1}$ is given as the form
{\small
\begin{align}\label{invmatXi}
\Xi(n,e)^{-1}=&\left[\begin{array}{*6{c}}
\displaystyle{e\choose e}&\displaystyle-{e+1\choose e}&\displaystyle{e+2\choose e}&\cdots&\displaystyle(-1)^{n-e-1}{n-1\choose e}&\displaystyle(-1)^{n-e}{n\choose e}\\
0&\displaystyle{e+1\choose e+1}&\displaystyle-{e+2\choose e+1}&\cdots&\displaystyle(-1)^{n-e-2}{n-1\choose e+1}&\displaystyle(-1)^{n-e-1}{n\choose e+1}\\
0&0&\displaystyle{e+2\choose e+2}&\cdots&\displaystyle(-1)^{n-e-3}{n-1\choose e+2}&\displaystyle(-1)^{n-e-2}{n\choose e+2}\\
\vdots&\vdots&\vdots&\ddots&\vdots&\vdots\\
0&0&0&\cdots&\displaystyle{n-1\choose n-1}&\displaystyle-{n\choose n-1}\\
0&0&0&\cdots&0&\displaystyle{n\choose n}
\end{array}\right]~.
\end{align}}

Therefore, 1-normalities of general linear sections $\mathbf{h}^{T}$ uniquely determine the tailing Betti numbers $\mathbf{b}^{T}$ and vice versa; we have exact formulae for the entries of $\mathbf{h}^{T}$ as follows:
\begin{align}\label{h_formula}
\trm{(Sectional 1-normality formula)}&\quad h^1(\mathcal I_{X\cap\Lambda^\alpha/\Lambda^\alpha}(1))=\ds\sum^{n}_{i=\alpha} (-1)^{i-\alpha}{i\choose \alpha}\beta_{i,2}(R/I_X)~,
\end{align}
where $e\le\alpha\le n$ and $\Lambda^{\alpha}$ is a general linear space of dimension $\alpha$. In particular, we could see that for some $\ell\ge e$
\begin{align*}
\trm{$\beta_{i,2}(R/I_X)=0$ for every $n\ge i\ge\ell$ $\Longleftrightarrow$ $h^1(\s I_{X\cap\Lambda^i/\Lambda^i}(1))=0$ for every $n\ge i\ge\ell$}~.
\end{align*}

Keeping in mind this correspondence, we consider the following corollaries:

\begin{Coro}\label{for_compt}
Let $X$ be any $3$-regular closed subscheme of codimension $e$ in $\P^n$ satisfying $\ND$ property. Let $\ell_0$ be the projective dimension of $R$-module $R/I_X$. Then, we have
\begin{itemize}
\item[(a)] (Rigidity on 2-regularity) The vanishing of $\beta_{e,2}(R/I_X)$ implies that $X$ is $2$-regular;
\item[(b)] (Lower bounds for tailing Betti numbers) Unless $X$ is $2$-regular, then for any $e\le i\le \ell_0$
\begin{align*}
\beta_{i,2}(R/I_X)&\ds\ge{\ell_0+1\choose i+1}~.
\end{align*}\end{itemize}
\end{Coro}
\begin{proof} (a) First, note that for any $\alpha\ge e$ and $i\ge 0$ we have 
\[H^i(\mathcal I_{X\cap\Lambda^{\alpha}/\Lambda^{\alpha}}(2-i))\to H^{i+1}(\mathcal I_{X\cap\Lambda^{\alpha+1}/\Lambda^{\alpha+1}}(1-i))\to H^{i+1}(\mathcal I_{X\cap\Lambda^{\alpha+1}/\Lambda^{\alpha+1}}(2-i))=0~,\]
for $X\cap\Lambda^{\alpha+1}$ is also $3$-regular. Thus, 
\begin{align}\label{coh_ineq}
 h^i(\mathcal I_{X\cap\Lambda^{\alpha}/\Lambda^{\alpha}}(2-i))&\ge h^{i+1}(\mathcal I_{X\cap\Lambda^{\alpha}/\Lambda^{\alpha}}(2-(i+1)))\\\nonumber
&\ge\cdots\\\nonumber
&\ge h^{i+n-\alpha}(\mathcal I_{X/\P^{n}}(2-(i+n-\alpha)))~.
\end{align}

By Formula (\ref{main_for}) and inequality (\ref{coh_ineq}), we know that
\[\ds\beta_{e,2}(R/I_X)\ge\sum^{n}_{\alpha=e} {\alpha\choose e}h^{1+n-\alpha}(\mathcal I_{X/\P^n}(2-(1+n-\alpha)))~,\]
a positive combination of cohomologies of some twists of the ideal sheaf $\mathcal I_X$. Hence, the vanishing $\beta_{e,2}(R/I_X)=0$ forces that
\[H^i(\mathcal I_{X/\P^n}(2-i))=0~\trm{for all $i=1,\cdots,\dim X+1$~,}\]
which implies $2$-regularity of $X$.

(b) By definition, it is obvious that $\beta_{\ell_0,2}(R/I_X)\neq 0$ and $\beta_{i,2}(R/I_X)=0$ for any $i\ge \ell_0+1$. So does $h^1(\mathcal{I}_{X\cap \Lambda^{i}/\Lambda^{i}}(1))$. Then, from Formula (\ref{main_for}) we have that for any $e\le i\le \ell_0$,
\[\ds\beta_{i,2}(R/I_X)=\sum_{\alpha=i}^{\ell_0} \binom{\alpha}{i}\,h^1(\mathcal I_{X\cap \Lambda^{\alpha}/\Lambda^{\alpha}}(1))\ge\sum^{\ell_0}_{\alpha=i} {\alpha\choose i}={\ell_0+1\choose i+1}~,\]
because $h^1(\mathcal I_{X\cap \Lambda^{\alpha}/\Lambda^{\alpha}}(1))\ge1$ for $i\le\alpha\le \ell_0$.
\end{proof}

\begin{Remk} The lower bound in Corollary \ref{for_compt} (b) is sharp: see Example \ref{final_ex} (c).
\end{Remk}

The tailing Betti numbers determine Hilbert polynomial of $X\subset\P^n$ precisely, so that we can read off some intrinsic invariants of $X$ such as \ti{arithmetic genus} in case of $X$ being a curve or \ti{irregularity} in case of $X$ being a surface from these Betti numbers.

Let us recall some basic facts. Say $r=\dim X$. It is well-known that there exist unique integers $\chi_r(X,\s O_X(1)),\cdots,\chi_0(X,\s O_X(1))$, which determine Hilbert polynomial of $X\subset\P^n$
\begin{equation}\label{hp_exp}P_X(t)=\sum^r_{j=0} \chi_j(X,\s O_X(1)){t+j-1\choose j}\quad\trm{or briefly, $\quad\sum^r_{j=0} \chi_j(X){t+j-1\choose j}$}\end{equation}
and that $\chi_j(X,\s O_X(1))=\chi_{j-1}(X\cap H,\s O_{X\cap H}(1))$ for each $1\le j\le r$, because we have the relation $\chi(\s O_X(t))-\chi(\s O_X(t-1))=\chi(\s O_{X\cap H}(t))$ from the exact sequence
\[0\to\s O_X(-1)\st{\cdot H}\s O_X\to \s O_{X\cap H}\to 0\quad\trm{($H$: a general hyperplane of $\P^n$)}~.\]
Thus, the determination of $P_X(t)$ is equivalent to decide all the $\chi_r(X),\cdots,\chi_0(X)$ in (\ref{hp_exp}).

\begin{Coro}\label{app_hilb}
Let $X$ be any $3$-regular $r$-dimensional connected algebraic set in $\P^n$ satisfying $\ND$ property. Set $e=n-r$, the codimension. Then, the tailing Betti numbers $\beta_{e,2}(R/I_X),\cdots,\beta_{n,2}(R/I_X)$ completely determine $P_X(t)$, Hilbert polynomial of $X\subset\P^n$ as follows:
\begin{align}\nonumber
P_X(t)&=\ds\left\{e+1+\sum^n_{i=e}(-1)^{i-e}{i\choose e}\beta_{i,2}(R/I_X)\right\}{t+r-1\choose r}\\
&+\sum^{r-1}_{i=0}\left[\left\{1-\beta_{n-i-1,2}(R/I_X)+\sum^n_{j=n-i}(-1)^{j-n+i}{j+1\choose n-i}\beta_{j,2}(R/I_X)\right\}{t+i-1\choose i}\right]~.
\end{align}
In particular, the degree of $X\subset\P^n$ can be given by
\begin{align}
e+1+\sum^n_{i=e}(-1)^{i-e}{i\choose e}\beta_{i,2}(R/I_X)
\end{align}
and the arithmetic genus $p_a(X)=(-1)^r(\chi(\s O_X)-1)$ can be computed by 
\begin{align}
(-1)^r\left\{(n+1)\beta_{n,2}(R/I_X)-\beta_{n-1,2}(R/I_X)\right\}~.
\end{align}
\end{Coro}

\begin{proof}
First of all, we recall \ti{sectional 1-normality formula} (\ref{h_formula}): for each $\alpha$ $(e\le\alpha\le n)$
\begin{align*}
h^1(\mathcal I_{X\cap\Lambda^\alpha/\Lambda^\alpha}(1))&=\ds\sum^{n}_{j=\alpha} (-1)^{j-\alpha}{j\choose \alpha}\beta_{j,2}(R/I_X)~,
\end{align*}
where $\Lambda^{\alpha}$ is a general linear space of dimension $\alpha$. To prove this assertion, it is enough to to decide all the coefficients $\chi_r(X),\cdots,\chi_0(X)$ with $\beta_{e,2}(R/I_X),\cdots,\beta_{n,2}(R/I_X)$.

For the top coefficient $\chi_r(X)$, we see that
\begin{align*}
\chi_r(X)&=\chi_0(X\cap\Lambda^e)=P_{X\cap\Lambda^e}(0)=\chi(\s O_{X\cap\Lambda^e})&\\
&=h^0(\s O_{X\cap\Lambda^e})=h^0(\s O_{X\cap\Lambda^e}(1))=h^0(\s O_{\Lambda^e}(1))+h^1(\s I_{X\cap\Lambda^e}(1))&\trm{by (\ref{h_formula})}\\
&\ds=e+1+\sum^n_{i=e}(-1)^{i-e}{i\choose e}\beta_{i,2}(R/I_X)~.&
\end{align*}
Next, for any $0\le i<r$, note that first $X\cap\Lambda^{n-i}$ is still connected and reduced by Bertini theorem so that $h^0(\s O_{X\cap\Lambda^{n-i}})=1,~h^1(\s I_{X\cap\Lambda^{n-i}})=0$ and that secondly $X\cap\Lambda^{n-i}$ is also $d$-regular for all $d\ge 3$. Thus, for any $0\le i<r$ we have 
\[0=H^1(\s I_{X\cap\Lambda^{n-i}})\to H^1(\s I_{X\cap\Lambda^{n-i}}(1))\to H^1(\s I_{X\cap\Lambda^{n-i-1}}(1))\to H^2(\s I_{X\cap\Lambda^{n-i}})\to H^2(\s I_{X\cap\Lambda^{n-i}}(1))=0~.\]

Using this, we calculate $\chi_i(X)$ as
\begin{align*}
\chi_i(X)&=\chi_0(X\cap\Lambda^{n-i})=\chi(\s O_{X\cap\Lambda^{n-1}})=h^0(\s O_{X\cap\Lambda^{n-i}})-h^1(\s O_{X\cap\Lambda^{n-i}})=1-h^2(\s I_{X\cap\Lambda^{n-i}})\\
&\ds=1-\left\{h^1(\s I_{X\cap\Lambda^{n-i-1}}(1))-h^1(\s I_{X\cap\Lambda^{n-i}}(1))\right\}\quad\trm{by (\ref{h_formula})}\\
&\ds=1-\left\{\sum^n_{j=n-i-1}(-1)^{j-n+i+1}{j\choose n-i-1}\beta_{j,2}(R/I_X)-\sum^n_{j=n-i}(-1)^{j-n+i}{j\choose n-i}\beta_{j,2}(R/I_X)\right\}\\
&\ds=1-\beta_{n-i-1,2}(R/I_X)+\sum^n_{j=n-i}(-1)^{j-n+i}{j+1\choose n-i}\beta_{j,2}(R/I_X)~.
\end{align*}

\end{proof}

In addition, we could express or bound the cohomologies $H^i(\mathcal I_{X/\P^n}(2-i))$ for $1\le i\le \dim(X)+1$, which is a measure on the \ti{deviation} of $X$ from $2$-regularity, with the tailing Betti numbers. For instance, for a small $i\le 3$, we can get simple formulae or bounds as below:

\begin{Coro}\label{app_coh}
Let $X$ be any $3$-regular $r$-dimensional connected algebraic set in $\P^n$ satisfying $\ND$ property. Set $e=n-r$, the codimension. Then, for given tailing Betti numbers $\beta_{e,2}(R/I_X),\cdots,\beta_{n,2}(R/I_X)$ we know that
\begin{itemize}
\item[(a)] $h^1(\mathcal I_{X/\P^n}(1))=\beta_{n,2}(R/I_X)$;
\item[(b)] $h^2(\mathcal I_{X/\P^n})=\beta_{n-1,2}(R/I_X)-(n+1)\beta_{n,2}(R/I_X)$;
\item[(c)] $h^3(\mathcal I_{X/\P^n}(-1))\ge\beta_{n-2,2}(R/I_X)-(n+1)\beta_{n-1,2}(R/I_X)+{n+2\choose 2}\beta_{n,2}(R/I_X)$ and

 \noindent $h^3(\mathcal I_{X/\P^n}(-1))\le\beta_{n-2,2}(R/I_X)-n\beta_{n-1,2}(R/I_X)+{n+1\choose 2}\beta_{n,2}(R/I_X)$;
\end{itemize}
\end{Coro}

\begin{proof} (a) it comes from part (b) of Proposition \ref{lem:20140113-1}. For (b) and (c), from the proof of Corollary \ref{app_hilb} we know that
\begin{align*}
&h^2(\s I_{X/\P^n})=h^1(\s I_{X\cap\Lambda^{n-1}/\Lambda^{n-1}}(1))-h^1(\s I_{X/\P^n}(1))~,\\
&h^2(\s I_{X\cap\Lambda^{n-1}/\Lambda^{n-1}})=h^1(\s I_{X\cap \Lambda^{n-2}/\Lambda^{n-2}}(1))-h^1(\s I_{X\cap\Lambda^{n-1}/\Lambda^{n-1}}(1))~.
\end{align*}
So, (b) is straightforward by Formula (\ref{h_formula}). For $h^3(\mathcal I_{X/\P^n}(-1))$, we have a similar sequence as
\[H^2(\s I_{X/\P^n}(-1))\to H^2(\s I_{X/\P^n})\to H^2(\s I_{X\cap\Lambda^{n-1}/\Lambda^{n-1}})\to H^3(\s I_{X/\P^n}(-1))\to H^3(\s I_{X/\P^n})=0~,\]
since $X$ is $3$-regular. Thus, we see that
\[h^1(\s I_{X\cap\Lambda^{n-1}/\Lambda^{n-1}})-h^2(\s I_{X/\P^n})\le h^3(\s I_{X/\P^n}(-1))\le h^2(\s I_{X\cap\Lambda^{n-1}/\Lambda^{n-1}})~,\]
so that lower and upper bounds in (c) can be obtained by Formula (\ref{h_formula}).
\end{proof}

\begin{Remk} Both lower and upper bounds for $h^3(\mathcal I_{X/\P^n}(-1))$ in Corollary \ref{app_coh} (c) are sharp. See Example \ref{final_ex} (b) for an optimal case of the lower bound and Example \ref{final_ex} (c) for the upper bound.
\end{Remk}

\begin{Ex}[with \texttt{Macaulay 2}]\label{final_ex} We check our results with some examples;
\begin{itemize}
\item[(a)] Let $\widetilde{C}$ be a rational normal curve in  $\P^{13}$ and $S^{m}(\widetilde{C})$ be the $m$-th higher secant variety of dimension $\min\{2m-1,13\}$. Now choose four points $q_1, q_2, q_3, q_4$ such that
\[q_1\in \P^{13} \setminus S^{6}(\widetilde{C}), \quad q_{2}, q_{3} \in S^{6}(\widetilde{C}) \setminus  S^{5}(\widetilde{C}), \quad q_4\in S^{5}(\widetilde{C}) \setminus  S^{4}(\widetilde{C})\] 
If we let $\Sigma =\langle q_1, q_2 , q_3, q_4\rangle$, then we know that $C=\pi_{\Sigma}(\widetilde{C})\subset \P^{9}$ is a smooth rational curve of degree 13 ($e=8$) and, using \texttt{Macaulay 2}, we can compute its $3$-regular Betti table
\[{\small\texttt{
\begin{tabular}{l|cccccccccccccccccccccccccccccc}
       & 0 & 1   & 2     & 3     & 4     & 5     & 6    & 7   & 8  & 9  \\[1ex]
\hline
    0 & 1 & .    & .      & .      & .      & .      &   .     & .      &  .     &  .    \\[1ex]
    1 & .  & 28 & 103 & 161 & 134 & 50   &   6    & .      &  .     &  .   \\[1ex]
    2 & .  &3  &  39  & 190 & 414 & 518 & 385  & 168 & \tbf{40}   & \tbf{4} \\[1ex]
    \end{tabular}
}}~,\]
where we denote the tailing Betti numbers as \tbf{boldface}. Now, we verify our theorems by this example. Since $\pi_{\Sigma}$ is an isomorphic projection from $\P^3$, we get that 1-normality $h^1(\mathcal I_{C/\P^9}(1))=4$ and $h^1(\mathcal I_{C\cap H/H} (1))=h^1(\mathcal O_{C\cap H} (1))-h^1(\mathcal O_{H} (1))=13-9=4$. Then, by Corollary \ref{main_for_cor} we can see that 
\[
\left[\begin{array}{cc}
\beta_{8,2}\\
\beta_{9,2}
\end{array}\right]
=\Xi(9,8)\cdot\left[\begin{array}{cc}
h^1(\mathcal I_{C\cap H/H} (1)) \\
h^1(\mathcal I_{C/\P^9}(1))
\end{array}\right]
=\left[\begin{array}{cc}
1&9 \\
0&1
\end{array}\right]\left[\begin{array}{cc}
4 \\
4
\end{array}\right]
=\left[\begin{array}{cc}
40\\
4
\end{array}\right]
~,\]
which coincides with tailing Betti numbers in the table above. Moreover, $\deg(C)$ can also be read off from $8+1+{8\choose 8}\beta_{8,2}-{9\choose 8}\beta_{9,2}=13$ and the (arithmetic) genus of $C$ by $\beta_{8,2}-{9+1\choose 9}\beta_{9,2}=0$ as Corollary \ref{app_hilb} says.

\item[(b)] Let $S$ be a smooth surface in $\P^8$ ($e=6$), which is Segre embedding of $C_1\times C_2$ where $C_1$ is a randomly chosen plane conic (i.e. smooth rational) curve and $C_2$ is a random plane cubic (i.e. elliptic) curve using \texttt{Macaulay 2}. The Betti table of $S$ is $3$-regualr and is given by
\[{\small\texttt{
\begin{tabular}{l|cccccccccccccccccccccccccccccc}
       & 0 & 1   & 2     & 3     & 4     & 5     & 6    & 7   & 8   \\[1ex]
\hline
    0 & 1 & .    & .      & .      & .      & .      & .  & .  &  .     \\[1ex]
    1 & .  & 15 & 40 & 45 & 24 & 5   &   .    & .      &  .    \\[1ex]
    2 & .  & 7 &  36  & 75 & 80 & 45 & \tbf{12}  & \tbf{1} & \tbf{0}   \\[1ex]
    \end{tabular}
}}~,\]
which shows that $S$ is linearly normal, but not ACM surface. By Degree formula in Corollary \ref{app_hilb}, we can check $\deg(S)=6+1+{6\choose6}\cdot\tbf{12}-{7\choose6}\cdot\tbf{1}+{8\choose6}\cdot\tbf{0}=12$. The irregularity $q(S)$ is also given using tailing Betti numbers as $\tbf{1}-(8+1)\cdot\tbf{0}=1$. In general, Hilbert polynomial can be given as
\[P_S(t)=12{t+1\choose 2}+3{t\choose 1}~,\]
by Corollary \ref{app_hilb}. Now, let us check the cohomology values $h^i(\s I_{S/\P^8}(2-i))$. We already know that $h^1(\s I_{S/\P^8}(1))=0, h^2(\s I_{S/\P^8})=h^1(\s O_{S/\P^8})=1$. What is $h^3(\s I_{S/\P^8}(-1))$? By lower and upper bounds in Corollary \ref{app_coh} (c), we can bound it as
\[3=\tbf{12}-(8+1)\cdot\tbf{1}+{8+2\choose2}\cdot\tbf{0}\le h^3(\s I_{S/\P^8}(-1))\le\tbf{12}-8\cdot\tbf{1}+{8+1\choose2}\cdot\tbf{0}=4~.\]
But, by Serre duality and K\"unneth formula, we know that 
\begin{align*}
h^3(\s I_{S/\P^8}(-1))&=h^2(\s O_{S/\P^8}(-1))=h^0(\omega_S(1))=h^0(\omega_{C_1}(1))\cdot h^0(\omega_{C_2}(1))\\
&=h^0(\s O_{C_1/\P^2}(2-3+1))\cdot h^0(\s O_{C_2/\P^2}(3-3+1))=1\cdot3=3~,
\end{align*}
and this shows that the lower bound in Corollary \ref{app_coh} (c) is sharp.
\item[(c)] Let $X\subset\P^{10}$ be a smooth Segre $5$-fold ($e=5$), which is a generic projection of the variety $\P^2\times \P^3\subset \P^{11}$ into $\P^{10}$. The degree of $X\subset \P^{10}$ is ${3+2\choose 2}=10$ and $X$ has a $3$-regular Betti tables of the projective dimension $\ell_0=10$ as follows:
\[
{\small\texttt{
 \begin{tabular}{c|ccccccccccccccccccccccccccccc}
        &0 & 1 & 2 & 3 & 4 & 5 & 6 & 7 & 8 & 9 & 10\\[1ex]
        \hline\\
        0 & 1 & 0   & 0 & 0 & 0 & 0 & 0 & 0 & 0 & 0 & 0 \\[1ex]
        1 & 0 & 6   & 0 & 0 & 0 & 0 & 0 & 0 & 0 & 0 & 0\\[1ex]
        2 & 0 & 20 &140& 331   & 471 & \tbf{465}   & \tbf{330} &\tbf{165}&\tbf{55}&\tbf{11}&\tbf{1}\\[1ex]        
 \end{tabular}
}}~.\]
First, we would like to mention that all the tailing Betti number $\beta_{i,2}$ except $\beta_{5,2}$ coincide with the lower bound ${\ell_0+1\choose i+1}$ in Corollary \ref{for_compt} (b) in this case; that is, this lower bound is optimal. Second, by Corollary \ref{main_for_cor}, we calculate 1-normalities of general linear sections
\[\small
\left[
\begin{array}{cccccccccccccc}
h^1(\mathcal I_{X\cap \Lambda^{5}/\Lambda^{5}}(1))\\
h^1(\mathcal I_{X\cap \Lambda^{6}/\Lambda^{6}}(1))\\
h^1(\mathcal I_{X\cap \Lambda^{7}/\Lambda^{7}}(1))\\
h^1(\mathcal I_{X\cap \Lambda^{8}/\Lambda^{8}}(1))\\
h^1(\mathcal I_{X\cap \Lambda^{9}/\Lambda^{9}}(1))\\
h^1(\mathcal I_{X/\P^{10}}(1))
\end{array}
\right]
=\Xi(10,5)^{-1}\cdot\left[
\begin{array}{cccccccccccccc}
\beta_{5,2}\\
\beta_{6,2}\\
\beta_{7,2}\\
\beta_{8,2}\\
\beta_{9,2}\\
\beta_{10,2}
\end{array}
\right]
=\left[
\begin{array}{cccccccccccccc}
1&-6& 21&-56&126&-252\\
0&1&-7&28 &-84 & 210\\
0&0&1& -8& 36& -120 \\
0&0&0&1&-9 & 45\\
0&0&0&0&1& -10 \\
0&0&0&0&0&  1
\end{array}
\right]
\left[
\begin{array}{cccccccccccccc}
465\\
330\\
165\\
55\\
11\\
1
\end{array}
\right]
=
\left[
\begin{array}{cccccccccccccc}
4\\
1\\
1\\
1\\
1\\
1
\end{array}
\right]
\]
and from these sectional 1-normalities we can compute many invariants of $X$; $\deg(X)$ is recovered by $h^0(\s O_{X\cap \Lambda^{5}}(1))=h^1(\s I_{X\cap \Lambda^{5}/\Lambda^{5}}(1))+h^0(\s O_{\Lambda^5}(1))=4+6=10$ and the \ti{sectional genus} of $X$, $g_s(X)$ can be obtained by $h^1(\s O_{X\cap \Lambda^{6}})=h^2(\s I_{X\cap \Lambda^{6}/\Lambda^{6}})=h^1(\s I_{X\cap \Lambda^{5}/\Lambda^{5}}(1))-h^1(\s I_{X\cap \Lambda^{6}/\Lambda^{6}}(1))=4-1=3$. More generally, by Corollary \ref{app_hilb} the Hilbert polynomial of $X$ can be computed as
\[P_X(t)=10{t+4\choose 5}-2{t+3\choose 4}+{t+2\choose 3}+{t+1\choose 2}+{t\choose 1}+1~,\]
which tells us $\chi(\s O_X)=1$ so that the irregularity $h^1(\s O_X)=0$. Finally, using Corollary \ref{app_coh} we check the cohomology values $h^1(\mathcal I_{X/\P^{10}}(1))=\tbf{1}, h^2(\mathcal I_{X/\P^{10}})=\tbf{11}-(10+1)\cdot\tbf{1}=0$, and $h^3(\mathcal I_{X/\P^{10}}(-1))\le\tbf{55}-10\cdot\tbf{11}+{10+1\choose 2}\cdot\tbf{1}=0$. Thus, $h^3(\mathcal I_{X/\P^{10}}(-1))=0$, which shows that the upper bound in Corollary \ref{app_coh} (c) is also sharp. 
\end{itemize}
\end{Ex}

\end{document}